\documentclass[12pt, a4paper, figuresright]{article}
\usepackage{amssymb, amsthm, amsfonts, mathrsfs}
\usepackage[T2A]{fontenc}
\usepackage[english]{babel}

\usepackage{xcolor}
\usepackage{commath}
\usepackage{thmtools}
\usepackage{thm-restate}
\usepackage{hyperref}
\usepackage{varioref}
\usepackage{cleveref}
\usepackage[numbers,square,sort&compress]{natbib}
\usepackage{makecmds}

\theoremstyle{plain}
\newtheorem{theorem}{Theorem}
\newtheorem{lemma}{Lemma}%
\newtheorem{corollary}{Corollary}

\theoremstyle{definition}%
\newtheorem{definition}{Definition}%

\theoremstyle{remark}%
\DeclareMathOperator{\RR}{\mathbb{R}}

\DeclareMathOperator{\ZZ}{\mathbb{Z}}

\DeclareMathOperator{\BC}{\mathcal{B}}

\DeclareMathOperator{\SC}{\mathcal{S}}

\DeclareMathOperator{\AC}{\mathcal{A}}
\DeclareMathOperator{\FC}{\mathcal{F}}
\DeclareMathOperator{\KC}{\mathcal{K}}

\DeclareMathOperator{\JC}{\mathcal{J}}
\DeclareMathOperator{\IC}{\mathcal{I}}

\DeclareMathOperator{\AS}{\mathscr{A}}
\DeclareMathOperator{\BS}{\mathscr{B}}

\DeclareMathOperator{\FS}{\mathscr{F}}

\DeclareMathOperator{\rank}{rank}

\DeclareMathOperator{\BUnit}{\mathbf 1}
\DeclareMathOperator{\BZero}{\mathbf 0}

\DeclareMathOperator{\supp}{supp}

\newcommand*{\intint}[2][1]{\left\{#1,\, \dots,\, #2\right\}}

{\par\color{red}} 
{} 

{\begin{quote}\color{blue}} 
{\end{quote}} 

{\par\color{green}} 
{} 

{\par\color{blue}} 
{}

{\par\color{blue}} 
{}

\newcommand\restr[2]{{
  \left.\kern-\nulldelimiterspace 
  #1 
  \vphantom{\big|} 
  \right|_{#2} 
  }}

\DeclareMathOperator{\VCdim}{\dim_{\!V\!C}}

\DeclareMathOperator{\ShiftHeller}{\widehat{h}}

\provideenvironment{section}[1]{\section{#1}\ignorespaces}{}

\provideenvironment{subsection}[1]{\subsection{#1}\ignorespaces}{}

\title{On the Column Number of Delta-Modular Matrices: an $O(m^2 \Delta)$ Bound via the Shifted Heller Constant}

\author{Elizaveta Pribytkova, Dmitrii Gribanov, \\ Dmitrii Taletskii, Stanislav Moiseev}
\date{\today}

\begin{document}

\maketitle

\begin{abstract}

In this paper, we build on the analysis initiated by Averkov and Schymura (2022)
and prove that the shifted Heller constant for matrices with $m$ rows is at most $m^2 + m$ for all sufficiently large $m$. Consequently, the number of distinct columns of a $\Delta$-modular matrix $A \in \mathbb{Z}^{m \times n}$ of rank $m$ is at most $O(m^2 \Delta)$. This improves upon the earlier bounds of $O(m^4 \Delta)$ and $O(m^2 \Delta^2)$.

\end{abstract}

\begin{section}
{Introduction}

    Our paper concerns the properties of so-called \emph{$\Delta$-modular} matrices. In particular, we are interested in the question:
    \begin{quote}
    	How many distinct columns can an arbitrary integer $\Delta$-modular matrix of rank $m$ have?
    \end{quote}
    
    We begin by introducing the necessary formal definitions. For a matrix $A \in \RR^{m \times n}$ of rank $m$ and $k \in \intint{m}$, let
    \[
    \Delta_k(A) = \max\left\{\abs{\det (A_{\IC \JC})} \colon \IC \in \binom{[m]}{k},\, \JC \in \binom{[n]}{k} \right\}
    \]
    denote the maximum absolute value of all $k \times k$ subdeterminants of $A$.\footnote{Here, $[n]$ denotes $\intint n$ and $\binom{[n]}{k}$ denotes the class of $k$-element subsets of $[n]$.}
    We also write $\Delta(A) := \Delta_m(A)$ for the maximum absolute value of the full-rank ($m \times m$) subdeterminants of $A$.
    
    Following the terminology of \cite{DiffColumnsOther}, a matrix $A$ is called $\Delta$-modular if $\Delta(A) = \Delta$, and $\Delta$-submodular if $\Delta(A) \le \Delta$.
    Moreover, a matrix \( A \) is called totally \( \Delta \)-modular if \( \max_k \Delta_k(A) = \Delta \), and totally \( \Delta \)-submodular if \( \max_k \Delta_k(A) \le \Delta \).
    
    Recall the main question of interest. The corresponding extremal quantity, commonly referred to as the \emph{generalized Heller constant} $h(\Delta, m)$, is the maximum integer $n > 0$ such that
    \begin{itemize}
        \item there exists a matrix $A \in \ZZ^{m \times n}$ with pairwise distinct columns,
        \item the matrix $A$ is $\Delta$-modular (in other words, $\Delta_m(A) = \Delta$).
    \end{itemize}
    Thus, the \emph{column number problem} consists in understanding the behavior of $h(\Delta,m)$ as a function of both the rank $m$ and the maximum subdeterminant bound $\Delta$.
    
    The study of matrices with bounded determinants is largely motivated by integer programming. The value \( \Delta(A) \) indicates how far the matrix is from being unimodular. In the case \(\Delta(A)=1\), all vertices of the LP relaxation are integral, and thus the ILP can be solved by a strongly polynomial-time algorithm. For \(\Delta(A)=2\), strongly polynomial-time algorithms are also known (see Artmann et al. \cite{BimodularStrong} and Veselov \& Chirkov \cite{BimodularVert}). Yet, for any fixed \(\Delta \geq 3\), the existence of polynomial-time algorithms for ILPs remains open. This motivates a thorough investigation of the structural characteristics of such matrices, especially the possible columns they may contain.
    
    A concrete and significant application of $h(\Delta,m)$ appears in Lee et al. \cite{ModularDiffColumns}. Namely, it is shown there that the $\ell_1$-norm distance $\pi$ between optimal solutions of the integer linear program (ILP) $\max\bigl\{c^\top x \colon A x = b,\; x \in \ZZ_{\geq 0}^n\bigr\}$ and those of its LP relaxation (commonly referred to as the \emph{$\ell_1$ ILP proximity}) satisfies
    \begin{equation*}
        \pi \leq (m+1) \cdot \Delta \cdot h(\Delta,m).
    \end{equation*}
    In the definition of the ILP, we assume that $A \in \ZZ^{m \times n}$, $b \in \ZZ^m$, and $\rank(A) = m$. This further underscores the importance of $h(\Delta,m)$ in ILP theory.
    
    
    
    
    
    
    The origin of the column-number problem lies in a classical result by I.~Heller \cite{Heller_original}, who established that
    \begin{equation}\label{eq:Heller}
        h(1,m)=m^2+m+1.
    \end{equation}
    Hence, the exact maximum number of distinct columns in the unimodular case is completely determined.
    
    For general \(\Delta\), Lee et al. \cite{ModularDiffColumns} established the first bound that is polynomial in both parameters
    \[
    h(\Delta,m)\le (m^2+m)\Delta^2+1,
    \]
    thereby improving upon the earlier result of Glanzer, Weismantel \& Zenklusen \cite{ModularRowsNum_Glanzer}, which was also quadratic in $m$.

    \noindent The results of Lee et al.~\cite{ModularDiffColumns} together with those of Heller \cite{Heller_original} yield
    \begin{equation}\label{eq:bimod_diff_colums}
        h(\Delta,m) = m^2 + (2 \Delta - 1)m + 1,
    \end{equation}
    whenever either $\Delta \in \{1,2\}$ (for arbitrary $m$) or $m \leq 2$ (for arbitrary $\Delta$).

    
    Via lattice theory, Averkov and Schymura \cite{DiffColumnsOther} revealed a remarkable relation between $h(\Delta,m)$ and the \emph{shifted Heller constant} $\ShiftHeller(m)$, expressed explicitly as follows:
    \begin{equation}\label{eq:Delta_shift}
        h(\Delta,m) \leq h(1,m) + (\Delta-1) \cdot \ShiftHeller(m).
    \end{equation}
    Here the \emph{shifted Heller constant} $\ShiftHeller(m)$ is defined as the maximal number \( n \) such that
    \begin{itemize}
        \item there exists a nonzero translation vector \( t \in [0, 1)^m \),
        \item and a matrix \( A \in \{-1, 0, 1\}^{m \times n} \) with pairwise distinct columns,
        \item such that \(\bigl( t \BUnit^\top + A \bigr)\) is totally $1$-submodular.
    \end{itemize}
    
    Using an argument based on the Sauer-Shelah-Perles Lemma, together with subtle reasoning about the existence of matrices of a certain type, Averkov and Schymura \cite{DiffColumnsOther} established that
    \begin{equation}\label{eq:ShiftedHeller_known_bounds}
        \ShiftHeller(m) \geq m(m-1) + 1, \quad\text{and}\quad \ShiftHeller(m) = O(m^4).
    \end{equation}
    Together with \eqref{eq:Delta_shift} and \eqref{eq:Heller}, \eqref{eq:ShiftedHeller_known_bounds} implies the main result of \cite{DiffColumnsOther}:
    \begin{equation*}
        h(\Delta,m) = O(m^4 \Delta).
    \end{equation*}
    Moreover, using a sharper version of $\ShiftHeller(m)$ (called the \emph{refined shifted Heller constant}), it is shown in \cite{DiffColumnsOther} that $h(\Delta,m) = O(m^3 \Delta)$ for odd $\Delta$.
    
    Our work relies heavily on the approach of Averkov and Schymura \cite{DiffColumnsOther}, which is based on the properties of the shifted Heller constant $\ShiftHeller(m)$, and on the exact results for the $(\leq\! 3)$-modular case established by Lee et al. \cite{ModularDiffColumns} and Paat, Walsh and Xu \cite{3ModularExact}. For further details, we refer the reader to \Cref{sec:contribution}.
    
    Let us now recall some other known facts regarding the behavior of $h(\Delta,m)$. The standard lower bound
    \begin{equation}\label{eq:hLB}
        h(\Delta,m) \geq m^2 + (2 \Delta - 1)m + 1
    \end{equation}
    is obtained from the matrix whose columns are the differences of pairs of vectors from the set
    \begin{equation*}
        \{\BZero,e_1, e_2, \dots, e_m\} \cup \{2 e_1, 3 e_1, \dots, \Delta e_1\}.
    \end{equation*}
    
    An important \emph{conjecture}, attributed to Lee et al. \cite{ModularDiffColumns} and Averkov and Schymura \cite{DiffColumnsOther}, asserts that
    \begin{equation*}
        h(\Delta,m) = \Theta(m^2 + m \Delta)
    \end{equation*}
    for sufficiently large values of either $\Delta$ or $m$. Progress in this direction was made in Geelen, Nelson \& Walsh \cite{ExcludingLineMatroids} and was substantially refined by Paat et al. \cite{DifferingSeparated}, who provided the following bound on the maximum number of non-parallel columns in $\Delta$-submodular matrices\footnote{The formal definition of this quantity is given in \Cref{def:parall_columns} (\Cref{sec:matrix_prelim}).}
    \begin{equation*}
        \binom{m+1}{2} + 80 \Delta^7 m, \quad \text{for $m \geq m_0(\Delta)$}.
    \end{equation*}

    \begin{subsection}
    {Our Contribution}\label{sec:contribution}
        We improve the best known bound on the shifted Heller constant $\ShiftHeller(m)$ in the following theorem.
        \begin{theorem}\label{th:main_shifted_Heller}
            \begin{equation*} 
                \ShiftHeller(m) \leq m^2 + m, \quad\text{for all sufficiently large } m.
            \end{equation*}
        \end{theorem}
        
        Together with the first inequality of \eqref{eq:ShiftedHeller_known_bounds}, this yields
        \begin{equation*}
            -m+1 \,\leq\; \ShiftHeller(m) - m^2 \;\leq\, m.
        \end{equation*}
        Furthermore, if $t$ is allowed to be the zero vector in the definition of the \emph{shifted Heller constant} — a version we denote by $\ShiftHeller_0(m)$ — then \Cref{th:main_shifted_Heller} together with \Cref{eq:Heller} yields
        \begin{equation*}
            \boxed{\ShiftHeller_0(m) = m^2 + m + 1, \quad\text{for all sufficiently large } m.}
        \end{equation*}
        
        Together with \eqref{eq:Delta_shift} and \eqref{eq:Heller}, \Cref{th:main_shifted_Heller} implies the following bound on $h(\Delta,m)$.
        \begin{corollary}
            The estimate
            \begin{equation*}
                h(\Delta,m) \leq \Delta \cdot \bigl(m^2 + m\bigr)+1
            \end{equation*}
            is valid for every $m \geq m_0$, where the threshold $m_0$ does not depend on $\Delta$.
        \end{corollary}

        To establish our results, we need a careful argument that relies on several structural facts obtained by Lee et al. \cite{ModularDiffColumns}, Paat, Walsh and Xu \cite{3ModularExact} and Averkov \& Schymura \cite{DiffColumnsOther}. Specifically, we combine the exact determination of the maximum number of nonparallel columns for $\Delta \leq 3$ (see \Cref{def:parall_columns} and the discussion after it) with the approach of Averkov and Schymura.
    \end{subsection}
\end{section}

\begin{section}{Preliminaries}
    \begin{subsection}{Notation}
        We denote vectors by lowercase letters $a,b,f$, matrices by uppercase letters $A,B,C$, sets by calligraphic letters $\AC,\BC,\FC$, and classes of sets by fancy letters $\AS,\BS,\FS$. We use $[n]$ to denote the set $\{1,2,\dots,n\}$ and $\binom{[n]}{k}$ to denote the class of $k$-element subsets of $[n]$.

        For $\FS \subseteq 2^{[n]}$ and $\AC \subseteq [n]$, we use the common notation for the \emph{trace of $\FS$ on $\AC$}
        \begin{equation*}
            \FS \cap \AC = \left\{ \FC \cap \AC \colon \FC \in \FS \right\}.
        \end{equation*}

        The $i$-th row and $j$-th column of a matrix $A$ are denoted by $A_{i*}$ and $A_{*j}$, respectively. For a subset of indices $\IC$, the corresponding subvector of a vector $v$ is written as $v_{\IC}$. Similarly, for subsets of indices $\IC$ and $\JC$, we define the row submatrix $A_{\IC *}$, the column submatrix $A_{* \JC}$, and the general submatrix $A_{\IC \JC}$.

        The support of a vector $v$, i.e., the set of indices of its nonzero components, is denoted by $\supp(v)$. The identity matrix of order $k$ is denoted by $I_k$. 
        The all-zero and all-ones column vectors are denoted by $\BZero$ and $\BUnit$, respectively; the dimensions of these vectors are clear from the context.
    \end{subsection}

    \begin{subsection}
    {Preliminaries from VC-theory}
        First, we recall some important notions from Vapnik-Chervonenkis theory (VC-theory for short). In particular, we use the concepts of \emph{shattering} and \emph{VC-dimension}.
        \begin{definition}
            Suppose that $\SC \subseteq [n]$ and $\FS \subseteq 2^{[n]}$. The class $\FS$ \emph{shatters} the set $\SC$ if
            \begin{equation*}
                \abs{\FS \cap \SC} = 2^{\abs{\SC}}.
            \end{equation*}
        \end{definition}



        \begin{definition}
            Suppose that $\FS \subseteq 2^{[n]}$. The \emph{VC-dimension} of $\FS$ is defined as the cardinality of a largest set that is shattered by $\FS$. We denote it by \[\VCdim(\FS).\]
        \end{definition}

        We now recall a fundamental result that is central to VC theory, namely the well-known Sauer--Shelah--Perles lemma \cite{Sauer1972, Shelah1972}, which was proved earlier, in a slightly weaker form, by Vapnik and Chervonenkis \cite{VapnikChervonenkis1971}.
        \begin{lemma}[Sauer-Shelah-Perles Lemma or Shattering Lemma]\label{lm:shatt}
            If $\FS \subseteq 2^{[n]}$ and $\VCdim(\FS) \leq k$, then
            \begin{equation*}
                \abs{\FS} \leq 
                \binom{n}{0} + \binom{n}{1} + \ldots + \binom{n}{k}.
            \end{equation*}
        \end{lemma}



    \end{subsection}

    \begin{subsection}
    {Preliminaries from the Theory of $\Delta$-modular Matrices}\label{sec:matrix_prelim}

        We now recall some additional facts from the theory of $\Delta$-modular matrices. This includes facts about matrices feasible under translations, as well as facts about the numbers of nonparallel columns.
        \begin{definition}
            Let $A \in \{-1,0,1\}^{m \times n}$. We say that a vector $t \in [0,1)^m$ is \emph{feasible} for $A$ if $\bigl(t \BUnit^\top + A\bigr)$ is totally $1$-submodular. Furthermore, we say that $A$ is \emph{feasible under translations} if there exists a vector $t \in [0,1)^m$ that is feasible for $A$; otherwise, $A$ is called \emph{infeasible under translations}.
        \end{definition}

        We recall a lemma from \cite{DiffColumnsOther} that allows us to estimate the number of columns of a matrix $A$ from the number of distinct supports of its columns.
        \begin{lemma}[Averkov and Schymura \cite{DiffColumnsOther}]\label{lm:columns_to_supps}
            Let $A \in \{-1,0,1\}^{m \times n}$ be a matrix with pairwise distinct columns and let $t \in [0,1)^m$ be a feasible vector for $A$. Let $\FS \subseteq 2^{[m]}$ be the class of supports of the columns of $A$. Then each support in $\FS$ accounts for at most two columns of $A$. Consequently,
            \[
                n \leq 2\abs{\FS}.
            \]
        \end{lemma}

        In fact, to obtain the sharpest bound on the shifted Heller constant, we will need a slight refinement of the previous lemma.
        \begin{lemma}\label{lm:columns_to_supps_extension}
            In addition to \Cref{lm:columns_to_supps}, the empty support occurs at most once, and so does any support containing an index $i$ with $0 < t_i < 1/2$.
        \end{lemma}
        \begin{proof}
            Suppose that a support contains an index $i$ with $0 < t_i < 1/2$, and assume that two distinct columns share this support. These columns must differ at some coordinate $j$ with $t_j = 0$, where their entries are $1$ and $-1$, respectively. The $2 \times 2$ minor of $\bigl(t \BUnit^\top + A\bigr)$ on rows $(i,j)$ then equals, in absolute value,
            \begin{equation*}
                \det \begin{pmatrix}
                    t_i - 1 & t_i - 1 \\
                    1 & -1
                \end{pmatrix} = 2(1 - t_i) > 1,
            \end{equation*}
            a contradiction.
        \end{proof}

        In addition, obtaining the sharpest estimates on the shifted Heller constant will require precise results on the number of nonparallel columns in $\Delta$-modular matrices.
        \begin{definition}\label{def:parall_columns}
            For $\Delta, m \geq 1$, define $s(\Delta,m)$ as maximum integer $n > 0$ such that
            \begin{itemize}
                \item there exists a matrix $A \in \ZZ^{m \times n}$ of rank $m$ with pairwise non-parallel columns,
                \item the matrix A is $\Delta$-submodular (in other words, $\Delta_m(A) \leq \Delta$).
            \end{itemize}
        \end{definition}

        \noindent The formula for $s(1,m)$ follows directly from Heller's result \eqref{eq:Heller}:
        \begin{equation}\label{eq:unimod_parall}
            s(1,m) = \binom{m+1}{2}.
        \end{equation}
        
        \noindent Oxley and Walsh \cite{BimodularMatrices_Oxley} determined $s(2,m)$ for sufficiently large $m$, and Lee et al.~\cite{ModularDiffColumns} independently proved the same formula for all $m \geq 6$:
        \begin{equation}\label{eq:bimod_parall}
            s(2,m) = \binom{m+1}{2} + m -1.
        \end{equation}
        
        \noindent Recently, Paat, Walsh and Xu \cite{3ModularExact} obtained the exact formula for the $3$-modular case, valid for all sufficiently large $m$:
        \begin{equation}\label{eq:threemod_parall}
            s(3,m) = \binom{m+1}{2} + 2(m-1).
        \end{equation}

    \end{subsection}
\end{section}

\begin{section}
{Proof of the Main Result (Theorem \ref{th:main_shifted_Heller})}
    Let $A \in \{-1,0,1\}^{m \times n}$ be a matrix with pairwise distinct columns and let $t \in [0,1)^m \setminus \{\BZero\}$ be a feasible translation for $A$. We need to show that $n \leq m^2 + m$ for all sufficiently large $m$.

    Denote $B = t \BUnit^\top + A$. Since $B$ is totally $1$-submodular, the condition $t_i > 0$ forces $A_{i j} \in \{-1, 0\}$ for all $j$. If $t_i > 1/2$, then the replacements
    \begin{equation*}
        t_i \;\longmapsto\; 1 - t_i, \qquad A_{i *} \;\longmapsto\; -\BUnit^\top - A_{i *}
    \end{equation*}
    negate the $i$-th row of $B$ and preserve both the absolute minor bounds and the pairwise distinctness of the columns. The new translation vector is still nonzero, and the new rows have entries in $\{-1,0\}$, since the original rows did. Hence, throughout the proof we may assume
    \begin{equation*}
        0 \leq t_i \leq \frac{1}{2} \quad (i \in [m]), \qquad t \not= \BZero.
    \end{equation*}

    To prove the main \Cref{th:main_shifted_Heller}, we work with supports of the columns of $A$ and use \emph{column slices} of $A$ along chosen values of row entries.
    Let $\FS \subseteq 2^{[m]}$ be the family of supports of the columns of $A$. In other words, 
    \begin{equation*}
        \FS = \left\{ \supp(A_{* j}) \colon j \in [n] \right\}.
    \end{equation*} For $a \in \{-1,0,1\}$ and $i \in [m]$, denote 
    \[ 
        \JC_{i}(a) = \left\{ j \in [n] \colon A_{i j} = a \right\}.
    \]
    We call $\JC_{i}(a)$ a \emph{slice of $A$}. The following \Cref{lm:shattered_triple} and \Cref{lm:Jslices} contain the two main ingredients.
    
    \begin{lemma}\label{lm:shattered_triple}
        If $t_i < 1/3$ for all $i \in [m]$, then $\VCdim (\FS) \leq 2$.
    \end{lemma}
    \begin{proof}
        Suppose, to the contrary, that some $3$-set is shattered by $\FS$; after permuting rows, it is $\{1,2,3\}$. Select a full-support column and three columns whose supports meet $\{1,2,3\}$ in $\{1, 2\}$, $\{1, 3\}$, $\{2, 3\}$ on these rows. Negate zero-translation rows (the rows with $t_i = 0$) if necessary to make the full-support column $(-1,-1,-1)^\top$. Positive entries can occur only in zero-translation rows, so this operation is permitted.

        Every pair-support column must have the same sign in both its nonzero positions. Indeed, if its entries on rows $i,j$ are $(1,-1)$, then $t_i = 0$. Consequently, the $2 \times 2$ minor of $B$ on rows $i,j$ formed by this column and the full-support column is 
        \begin{equation*}
            \det \begin{pmatrix}
                1 & -1\\
                t_j -1 & t_j - 1
            \end{pmatrix} = 2(t_j-1) < - 4/3.
        \end{equation*}
        A positive pair $(1,1)$ forces both translations to vanish.

        If at least two translations are positive, all three pair-support columns are negative. In this case, writing $(x,y,z) = (t_1, t_2, t_3)$, the $3 \times 3$ minor of $B$ on rows $(1,2,3)$ and these columns is
        \begin{equation*}
            \det \begin{pmatrix}
                x-1 & x-1 & x \\
                y-1 & y & y-1 \\
                z & z-1 & z-1
            \end{pmatrix} = 2 - x - y -z.
        \end{equation*}
        Since $x + y + z < 1$, this contradicts the feasibility of $t$.

        If exactly one coordinate is positive, call it $u$ and put it first; note that $0 < u < 1/3$. The two pairs meeting that coordinate are negative. If the third pair is negative, the minor is
        \begin{equation*}
            \det \begin{pmatrix}
                u-1 & u-1 & u \\
                -1 & 0 & -1 \\
                0 & -1 & -1
            \end{pmatrix} = 2 - u > 1;
        \end{equation*}
        
        \noindent If it is positive, it is
        \begin{equation*}
            \det \begin{pmatrix}
                u-1 & u-1 & u \\
                -1 & 0 & 1 \\
                0 & -1 & 1
            \end{pmatrix} = 3u - 2 < -1.
        \end{equation*}

        \noindent If all translations vanish, the minor is
        \begin{equation*}
            \pm \det \begin{pmatrix}
                1 & 1 & 0 \\
                1 & 0 & 1 \\
                0 & 1 & 1
            \end{pmatrix} = \mp 2.
        \end{equation*}
        
        \noindent Thus, no triple can be shattered, yielding $\VCdim (\FS) \leq 2$.
    \end{proof}

    Using \Cref{lm:shattered_triple}, we directly obtain the bound $n \leq m^2 + m + 2$ when $t_i < 1/3$ for all $i \in [m]$. Indeed, \Cref{lm:shattered_triple} and \Cref{lm:shatt} give $\abs{\FS} \leq \binom{m}{0} + \binom{m}{1} + \binom{m}{2}$, and \Cref{lm:columns_to_supps} then implies $n \leq 2\abs{\FS} \leq m^2 + m + 2$. Nevertheless, a slightly better bound can be achieved by exploiting $t \not= \BZero$ and \Cref{lm:columns_to_supps_extension}.
    \begin{corollary}\label{cor:shattered_triple}
        If $t_i < 1/3$ for all $i \in [m]$, then $n \leq m^2 + 2$.
    \end{corollary}
    \begin{proof}
        Since $t \not= \BZero$, we can choose an index $i$ with $0 < t_i < 1/3$ and set
        \begin{equation*}
            \FS_0 = \left\{\FC \in \FS \colon i \notin \FC \right\}.
        \end{equation*}
        By \Cref{lm:columns_to_supps_extension}, each support occurs at most twice, while supports containing $i$ occur only once. Hence $n \leq \abs{\FS} + \abs{\FS_0}$. Moreover, \Cref{lm:shattered_triple} gives $\VCdim(\FS) \leq 2$ and $\VCdim(\FS_0) \leq 2$. Since $\FS_0 \subseteq 2^{[m] \setminus \{i\}}$, \Cref{lm:shatt} yields
        \begin{multline*}
            n \leq \left(\binom{m}{0} + \binom{m}{1} + \binom{m}{2}\right) + \\
            \left(\binom{m-1}{0} + \binom{m-1}{1} + \binom{m-1}{2}\right) = m^2 + 2.
        \end{multline*}
    \end{proof}

    The following lemma provides a powerful tool for working with slices.
    \begin{lemma}\label{lm:Jslices}
        Let $i \in [m]$ with $t_i > 0$ and $a \in \{-1,0\}$. Set $c = t_i + a$ and $c' = \left\lfloor 1/\abs{c} \right\rfloor$. Then,
        \begin{equation*}
            \abs{\JC_i(a)} \leq s(c',m) -m +1.   
        \end{equation*}
    \end{lemma}
    \begin{proof}
        Note that $c \not= 0$ and $\abs{c} < 1$. After a permutation of the rows, we may assume that $i = 1$. If the slice is empty, the bound holds since $s(k, m) \geq m$ for $k \geq 1$. Otherwise, form the integer matrix
        \begin{equation*}
            E_a = \begin{pmatrix}
                \BUnit^\top & \BZero^\top \\
                A_{-1 \JC_i(a)} & I_{m-1}
            \end{pmatrix}.
        \end{equation*}
        Here $A_{-1 \JC_i(a)}$ denotes the submatrix of $A$ lying in the intersection of the rows indexed by $[m]\setminus \{1\}$ and columns indexed by $\JC_i(a)$. Denote the first $\abs{\JC_i(a)}$ columns of $E_a$ by $D_a$. Note that the columns of $D_a$ are pairwise non-parallel and have first coordinate one. The last $(m-1)$ columns of $E_a$ have first coordinate zero, are independent, and are non-parallel to the first columns. The matrix $E_a$ has rank $m$.

        Expand an arbitrary $m \times m$ minor of $E_a$ along all selected columns from $I_{m-1}$. The remaining determinant, up to sign, is a square minor of $D_a$ containing its first row. If it has order $k \times k$, use the corresponding $k$ columns of $B$, the first row of $B$, and the corresponding $(k-1)$ other rows of $B$. Factoring out $c$ and then subtracting $t_l$ times the first row from each other $l$-th row, we get
        \begin{equation*}\label{eq:tl_subtracting}
            \abs{\det B_{\IC \KC}} = \abs{c} \cdot \abs{\det (D_a)_{\IC \KC}}.
        \end{equation*}
        Here $\IC$ denotes a set of $k$ row indices containing $1$ and $\KC$ denotes a set of $k$ column indices. As an example of this operation, consider the minor of $D_a$ of order $k = 3$. Assume that $\IC = \{1,2,3\}$ and $\KC = \{1,2,3\}$. For the corresponding minor of $B$ we get
        \begin{multline*}
            \abs{\det B_{\IC \KC}} = \abs{\det \begin{pmatrix}
                c & c & c \\
                t_2 + A_{2 1} & t_2 + A_{2 2} & t_2 + A_{2 3} \\
                t_3 + A_{3 1} & t_3 + A_{3 2} & t_3 + A_{3 3}
            \end{pmatrix}} = \\
            \abs{c} \cdot \abs{\det \begin{pmatrix}
                1 & 1 & 1 \\
                t_2 + A_{2 1} & t_2 + A_{2 2} & t_2 + A_{2 3} \\
                t_3 + A_{3 1} & t_3 + A_{3 2} & t_3 + A_{3 3}
            \end{pmatrix}} = \abs{c} \cdot \abs{\det \begin{pmatrix}
                1 & 1 & 1 \\
                A_{2 1} & A_{2 2} & A_{2 3} \\
                A_{3 1} & A_{3 2} & A_{3 3}
            \end{pmatrix}} = \\ \abs{c} \cdot \abs{\det \bigl(D_a\bigr)_{\IC \KC}}.
        \end{multline*}
        
        Since $B$ is totally $1$-submodular, the integer $\abs{\det (D_a)_{\IC \KC}}$ is at most $1/\abs{c}$, hence at most $c' = \left\lfloor 1/\abs{c} \right\rfloor$. Every $m \times m$ minor of $E_a$ obeys this bound. And $E_a$ has $\abs{\JC_i(a)} + m -1$ pairwise non-parallel columns. Thus, $\abs{\JC_i(a)} + m -1 \leq s(c',m)$, which proves the lemma.
    \end{proof}

    {\bf Now, we can finish the proof of \Cref{th:main_shifted_Heller}.} If $t_i < 1/3$ for every $i \in [m]$, then \Cref{cor:shattered_triple} gives $n \leq m^2 + 2$. Since $m^2 + 2 \leq m^2 + m$ for $m \geq 2$, we obtain $n \leq m^2 + m$.

    Now, assume that there exists $i$ such that $t_i \geq 1/3$. Considering the slices along the $i$-th row, since $t_i > 0$ and $\JC_i(1) = \emptyset$, we can write $n = \abs{\JC_i(-1)} + \abs{\JC_i(0)}$. Next, we apply \Cref{lm:Jslices} with $a = -1$ and $a = 0$ in each of three cases: $t_i = 1/3$, $t_i = 1/2$ and $1/3 < t_i < 1/2$.
    \begin{itemize}
        \item {\bf The case $t_i = 1/3$.} We get $n \leq s(1,m) + s(3,m) - 2m + 2$.
        By \Cref{eq:unimod_parall} and \Cref{eq:threemod_parall}, one gets
        \begin{equation*}
            n \leq \binom{m+1}{2} + \left(\binom{m+1}{2} + 2(m-1)\right) - 2m + 2 = m^2 + m
        \end{equation*}
        for sufficiently large $m$;

        \item  {\bf The case $t_i = 1/2$.} We get $n \leq 2s(2,m) - 2m + 2$. By \Cref{eq:bimod_parall}, one gets
        \begin{equation*}
            n \leq 2\left(\binom{m+1}{2}+m-1\right) - 2m + 2 = m^2 + m
        \end{equation*}
        for sufficiently large $m$;

        \item  {\bf The case $1/3 < t_i < 1/2$.} We get $n \leq s(1,m) + s(2,m) - 2m + 2$. By \Cref{eq:unimod_parall} and \Cref{eq:bimod_parall}, one gets
        \begin{equation*}
            n \leq \binom{m+1}{2} + \left(\binom{m+1}{2}+m-1\right) - 2m + 2 = m^2 +1 \leq m^2 + m
        \end{equation*}
        for sufficiently large $m$.
    \end{itemize}
    Each of the considered cases satisfies the inequality $n \leq m^2 + m$  for sufficiently large $m$. This proves \Cref{th:main_shifted_Heller}.
\end{section}

























\section*{Acknowledgements}

The authors would like to express their sincere gratitude to the organizers of the Research Experience Program for Undergraduates 2026 (LIPS) for providing an excellent research platform and continuous support. The program was hosted by the Laboratory of Combinatorial and Geometric Structures at the Phystech School of Applied Mathematics and Informatics, MIPT. We also thank our fellow participants for many insightful conversations that contributed to this work.

\section*{Statements and Declarations}

\begin{itemize}
    \item {\bf Competing interests.} The authors declare that they have no competing interests.

    \item {\bf Data availability.} No new data were created or analysed in this study. Data sharing is not applicable to this article.

    \item {\bf Use of artificial intelligence.} The proof was obtained with the help of ChatGPT proving tools during work on the previous version (\cite[version~3]{work_prev_version}) of the text. ChatGPT helped substantially improve the results that were originally obtained by the authors. The authors take full responsibility for the final content and results.
\end{itemize}
    
\bibliographystyle{plainnat}
\bibliography{biblio}

\end{document}